\documentclass[10pt]{amsart}
\usepackage{a4,tikz,enumerate,dsfont,fancyhdr,colonequals,mathrsfs}
\usepackage[pdftex]{hyperref}
\numberwithin{equation}{section}
\newtheorem{theo}{Theorem}
\newtheorem{lemma}[theo]{Lemma}
\newtheorem{prop}[theo]{Proposition}
\newtheorem{cor}[theo]{Corollary}
\newtheorem{defi}[theo]{Definition}
\newtheorem{assum}[theo]{Assumptions}
\theoremstyle{definition}
\newtheorem{rem}[theo]{Remark}
\newtheorem{rems}[theo]{Remarks}

\numberwithin{theo}{section}

\author{Delio Mugnolo}

\thanks{This work has been written while the author was supported by the Istituto Nazionale di Alta Matematica ``Francesco Severi'' as a Ph.D. student at the University of T\"ubingen. The author thanks his supervisor, Prof. Rainer Nagel, for countless motivating discussions.}
\address{Dipartimento di Matematica dell'Universit{\`a} degli Studi, Via Orabona 4, I-70125 Bari, Italy}
\email{mugnolo@dm.uniba.it}
\title[Operator matrices and cosine functions]{Operator matrices as generators of cosine operator functions}

\begin{document}

\begin{abstract}
We introduce an abstract setting that allows to discuss wave equations with time-dependent boundary conditions by means of operator matrices. We show that such problems are well-posed if and only if certain perturbations of the same problems with homogeneous, time-independent boundary conditions are well-posed. As applications we discuss two wave equations in $L^p(0,1)$ and in $L^2(\Omega)$ equipped with dynamical and acoustic-like boundary conditions, respectively.
\end{abstract}

\maketitle

\section{Introduction}

The theory of (undamped) second order abstract Cauchy problems has been studied for a long time. To that purpose, cosine operator functions and first order reductions have been introduced already in the 1960s (see \cite{[So66]}, \cite{[DG67]}, \cite{[Fa68]} and \cite{[Fa68b]}, \cite{[Go69b]}). In order to use these abstract theories for wave equations on bounded domains, homogeneous time-independent boundary conditions have most frequently been considered so far. For classical results we refer to \cite{[Fa85]} and \cite{[XL98]} (see also \cite[\S~II.7]{[Go85]} and \cite[Chapt.~7]{[ABHN01]}). However, time-dependent boundary conditions also occur in many situations. We mention in particular wave equations with acoustic boundary conditions (see~\cite{[Mu04]} and references therein), boundary contact problems (see~\cite{[Be00]} and references therein), and mechanical systems that can be modeled as second order problems with first order dynamical conditions (see~\cite{[CENP04]}). 

\bigskip
In this paper we choose an abstract approach and study second order abstract initial-boundary value
problems with {\sl dynamical boundary conditions} of the form  
\begin{equation}\tag{AIBVP$_2$}
\left\{
\begin{array}{rcll}
 \ddot{u}(t)&=& {A}u(t), &t\in{\mathbb R},\\
 \ddot{x}(t)&=& Bu(t)+ \tilde{B} x(t), &t\in {\mathbb R},\\
  x(t)&=&Lu(t), &t\in {\mathbb R},\\
  u(0)&=&f\in X,\qquad\;\;  \dot{u}(0)=g\in X,&\\
  x(0)&=&h\in\partial X,\qquad  \dot{x}(0)=j\in \partial X,&
\end{array}
\right.
\end{equation}
on a Banach \emph{state space} $X$, and a Banach {\emph boundary space} $\partial X$.

\bigskip
Our aim is to characterize the well-posedness of $\rm(AIBVP_2)$ -- or equivalently the property of cosine operator function generator of the operator matrix
\begin{equation*}
\mathcal{A}:=\begin{pmatrix}
A & 0\\
B & \tilde{B}
\end{pmatrix}
\end{equation*}
with non-diagonal domain
\begin{equation*}
D(\mathcal{A}):=\left\{ 
\begin{pmatrix}
u\\ x
\end{pmatrix}\in D(A)\times\partial X : Lu=x
\right\}
\end{equation*}
on the product space $X\times \partial X$ -- by means of properties of the operators $A,B,\tilde{B},L$ involved. As our main result we show that the well-posedness of $\rm(AIBVP_2)$ is equivalent to the well-posedness of a certain (possibly perturbed) second order problem on $X$ with homogeneous (e.g., in concrete applications, Dirichlet or Neumann) boundary conditions, cf. Theorems~\ref{main} and~\ref{main2} below and subsequent remarks.

\bigskip
The motivation for considering $\rm(AIBVP_2)$ is threefold. First, there is a deep relation between  wave equations equipped with second order dynamical boundary conditions and wave equations equipped with acoustic boundary conditions -- this has been convincingly shown in~\cite{[GGG03]}. In particular, showing well-posedness and compactness properties for one problem also yields analogous results for the other.

Further, if we combine the first and the third equations in $\rm(AIBVP_2)$ we obtain
\begin{equation}\label{mot}
\ddot{x}(t)=LAu(t),\qquad t\in\mathbb{R},
\end{equation}
provided that the second derivative with respect to time and the operator $L$ commute -- which in applications is often the case, at least for smooth $u$'s. Plugging~\eqref{mot} into the the second equation of $\rm(AIBVP_2)$ what we obtain is
\begin{equation*}
LAu(t)=Bu(t)+\tilde{B}Lu(t),\qquad t\in\mathbb{R},
\end{equation*}
an abstract version of the so-called generalized Wentzell (or Wentzell-Robin) boundary conditions for the operator $A$, cf.~\cite{[En03]}. Elliptic operators equipped with such kind of boundary conditions have received vivid interest in the last years -- for results in $L^p$-spaces cf.~\cite{[FGGR02]}, \cite{[AMPR03]}, \cite{[MR04]}, and~\cite{[VV03]}, where also an interesting probabilistic interpretation is given.

Finally, $\rm(AIBVP_2)$ is the natural second order version of the abstract initial-boundary value problem considered, e.g., in~\cite{[CENN03]} and~\cite{[KMN03]}. If we show that $\rm(AIBVP_2)$ is well-posed -- i.e., that $\mathcal A$ generates a cosine operator function --, then by~\cite[Thm.~3.14.17]{[ABHN01]} we also obtain that $\mathcal A$ generates an analytic semigroup of angle $\frac{\pi}{2}$. In particular, this yields well-posedness of the first order version of $\rm(AIBVP_2)$, i.e.,
\begin{equation}\tag{AIBVP}
\left\{
\begin{array}{rcll}
 \dot{u}(t)&=& {A}u(t), &t\geq 0,\\
 \dot{x}(t)&=& Bu(t)+ \tilde{B} x(t), &t\geq 0,\\
  x(t)&=&Lu(t), &t\geq 0.\\
  u(0)&=&f\in X,&\\
  x(0)&=&h\in \partial X.&\\
\end{array}
\right.
\end{equation}
While it is known that the product space $X\times \partial X$ is the right framework to discuss well-posedness of $\rm(AIBVP)$ (see~\cite{[CENN03]}), a surprising result of our paper is that the phase space associated to $\rm(AIBVP_2)$ need \emph{not} be a product space, cf. Section~5.

\bigskip
As for $(\rm{AIBVP})$ (see \cite{[CENN03]} and~\cite{[EF04]} for the cases of $L\notin{\mathcal L}(X,\partial X)$ and $L\in{\mathcal L}(X,\partial X)$, respectively), we need to distinguish three different cases: the so-called {\sl boundary operator} $L$ can be unbounded from some \emph{Kisy\'nski space} (see Definition~\ref{wellp} below) $Y$ to $\partial X$; unbounded from $X$ to $\partial X$ but bounded from $Y$ to $\partial X$; or bounded from $X$ to $\partial X$. In this paper \emph{we only consider the first two cases} in Sections~4 and 5, respectively. These occur, e.g., when we consider a wave equation on an $L^p$-space and $L$ is the normal derivative or the trace operator, respectively.

\bigskip
The third case (i.e., $L\in{\mathcal L}(X,\partial X)$) is typical for wave equations equipped with Wentzell boundary conditions on spaces where the point evaluation is a bounded operator. Among those who have already treated such problems in $C([0,1])$ we mention Favini, G.R. Goldstein, J.A. Goldstein, and Romanelli (\cite{[FGGR01]}), who considered plain Wentzell boundary conditions for the second derivative, and Xiao and Liang (\cite{[XL03]}), who treated {\sl generalized} Wentzell boundary conditions. Later, B\'atkai, Engel, and Haase extended the above results to second order problems involving (possibly degenerate) elliptic operators and (possibly {\sl non-local}) generalized Wentzell boundary conditions on $C[0,1]$ and $W^{1,1}(0,1)$ in~\cite{[BE04]} and~\cite{[BEH04]}, respectively. Finally, Engel (\cite[\S~5]{[En04]}) developed an abstract framework that includes all the above mentioned results as special cases.

\bigskip
While in this paper we focus mainly on the well-posedness of $(\rm{AIBVP}_2)$, in Section~6 we briefly discuss compactness, regularity and asymptotic issues. 

In a general Banach space setting, an efficient criterion to check the boundedness of a cosine operator function is still missing. However, if we \emph{assume} the cosine and sine operator function governing $(\rm{AIBVP}_2)$ to be bounded (i.e., if we assume the solutions to such a problem to be bounded in time), then we can make use of the spectral theory for operator matrices developed in~\cite{[En99]} and~\cite{[KMN03]}. In this way we are able to obtain sufficient conditions for the (almost) periodicity of solutions to $(\rm{AIBVP}_2)$.

\bigskip
As an application of our theory, in Section~7 we generalize some known results on a concrete second order problem. 

Lancaster, Shkalikov, and Ye (\cite[\S~5 and \S~7]{[LSY93]}) and later Gal, G.R. Goldstein and J.A. Goldstein (\cite{[GGG03]}), and Kramar, Nagel, and the author (\cite[Rem.~9.13]{[KMN03b]}), have already considered wave equations with second order dynamical boundary conditions in an $L^2$-setting, but used quite different methods. The following is a corollary of statements obtained in these papers, where all the proofs deeply rely on the Hilbert space setting. We quote it as a motivation for our investigations.

\begin{prop}\label{motivp}
The problem
\begin{equation}\label{motiv}
\left\{
\begin{array}{rcll}
 \ddot{u}(t,x)&=& u''(t,x), &t\in {\mathbb R},\; x\in(0,1),\\
 \ddot{u}(t,j)&=& (-1)^j u'(t,j)+\beta_j u(t,j), &t\in {\mathbb R},\; j=0,1,\\
  u(0,\cdot)&=&f,\qquad  \dot{u}(0,\cdot)=g,&
\end{array}
\right.
\end{equation}
admits for all $f,g\in H^2(0,1)$ and $\beta_0,\beta_1\in \mathbb C$ a unique classical solution $u$, continuously depending on the initial data. If $(\beta_0, \beta_1)\in{\mathbb R}_-^2 \setminus\{0,0\}$, then such a solution is uniformly bounded in time with respect to the $L^2$-norm.
\end{prop}

Applying our technique, in Proposition~\ref{motivrev} we are able to show that~\eqref{motiv} is well-posed in a general $L^p$-setting, and to describe the associated phase space. Moreover, we show that the solution $u$ is a smooth function if $f,g$ are smooth, too. We also discuss (almost) periodicity in two special cases.

Further, as a second application we consider in Proposition~\ref{appl2} a wave equation equipped with acoustic-like boundary conditions on $L^2(\Omega)$, $\Omega\subset \mathbb{R}^n$.

\bigskip
We start in Section~2 by summarizing some results from the theory of cosine operator functions. Some of them are new and of independent interest. 

Our main technique will be based on operator matrices with coupled (i.e., non-diagonal) domain, which have already proved a powerful tool to tackle first order problems, cf.~\cite{[CENN03]}	. We recall a few main results from this theory in Section~3, and refer the reader to~\cite{[KMN03b]} for a more thorough introduction to this theory. 

\section{General results on cosine operator functions}

To keep the paper as self-contained as possible, we first recall the definition and some basic properties of a cosine operator function.

\begin{defi}\label{cof}
Let $E$ be a Banach space. A strongly continuous function 
$C:{\mathbb R}\to {\mathcal L}(E)$ is called a \emph{cosine operator function} if it satisfies the D'Alembert functional relations
$$\left\{
\begin{array}{rcl}
C(t+s)+C(t-s)&=&2C(t)C(s),\qquad t,s\in\mathbb R,\\
C(0)&=&I_E.
\end{array}
\right.$$
Further, the operator $K$ on $E$ defined by 
$$Kx:=\lim_{t\to 0}\frac{2}{t^2}(C(t)x-x),\qquad
D(K):=\left\{x\in E: \lim_{t\to 0}\frac{2}{t^2}(C(t)x-x)\;\hbox{exists}\right\},$$
is called the \emph{generator of} $(C(t))_{t\in\mathbb R}$, and we denote $C(t)=C(t,K)$, $t\in\mathbb R$. We define the associated \emph{sine operator function} $(S(t,K))_{t\in\mathbb R}$ by 
$$S(t,K)x:=\int_0^t C(s,K)x ds,\qquad t\in{\mathbb R},\; x\in E.$$
\end{defi}

The relation between well-posedness for second order abstract Cauchy problems and cosine operator functions is very close to that between first order abstract Cauchy problem and strongly continuous semigroups, as explained in the following, cf.~\cite[Thm.~3.14.11]{[ABHN01]}. 

\begin{lemma}\label{wellp2char}
Let $K$ be a closed\footnote{For a closed operator $K$ on a Banach space, $[D(K)]$ will denote throughout this paper the Banach space obtained by equipping its domain with the graph norm.} operator on a Banach space $E$. Then the operator $K$ generates a cosine operator function on $E$ if and only if there exists a Banach space $F$, with dense imbeddings $[D(K)]\hookrightarrow F\hookrightarrow E$, such that the operator matrix 
\begin{equation}\label{reduxm}
{\mathbf K}:=
\begin{pmatrix}
0 & I_F\\
K & 0
\end{pmatrix},\qquad D({\mathbf K}):=D(K)\times F,
\end{equation}
generates a strongly continuous semigroup $(e^{t\mathbf K})_{t\geq 0}$ in $F\times E$. In this case $F$ is uniquely determined and coincides with the space of strong differentiability of the operator valued mapping $C(\cdot,K):\mathbb{R}\to{\mathcal L}(E)$, and there holds
\begin{equation}\label{groupcos}
e^{t{\mathbf K}}=
\begin{pmatrix}
C(t,K) & S(t,K)\\
KS(t,K) & C(t,K)
\end{pmatrix},\qquad t\geq 0.
\end{equation}
\end{lemma}

\begin{defi}\label{wellp}
If $K$ generates a cosine operator function on $E$, then the subspace $F$ of $E$ introduced in Lemma~\ref{wellp2char} is called \emph{Kisy\'nski space associated with $(C(t,K))_{t\in{\mathbb R}}$}. The product space ${\bf E}=F\times E$ is called \emph{phase space associated with} $(C(t,K))_{t\in{\mathbb R}}$.
\end{defi}

Taking into account Lemma~\ref{wellp2char} we can reformulate some known similarity and perturbation results semigroups in the context of cosine operator functions.

\begin{lemma}\label{simil}
Let $E_1, E_2, F_1, F_2$ be Banach spaces, with $F_1\hookrightarrow E_1$ and $F_2\hookrightarrow E_2$, and let $U$ be an isomorphism from $F_1$ onto $F_2$ and from $E_1$ onto $E_2$. Then an operator $K$ generates a cosine operator function with associated phase space $F_1\times E_1$ if and only if $UKU^{-1}$ generates a cosine operator function with associated phase space $F_2\times E_2$. In this case 
\begin{equation}\label{similrel}
U C(t,K) U^{-1}= C(t, UKU^{-1}),\qquad t\in\mathbb{R}.
\end{equation}
\end{lemma}

\begin{proof} 
Since the operator matrix
$${\bf U}:= \begin{pmatrix}
U & 0\\
0 & U
\end{pmatrix}$$
is an isomorphism from ${\bf E}_1:=F_1\times E_1$ onto ${\bf E}_2:=F_2\times E_2$, it follows by similarity that the reduction matrix ${\bf K}$ defined in~\eqref{reduxm} generates a strongly continuous semigroup on ${\bf E}$ if and only if ${\bf UKU}^{-1}$ generates a strongly continuous semigroup $(e^{t{\bf UKU}^{-1}})_{t\geq 0}$  on ${\bf F}$. Now
$${\bf UKU}^{-1}=\begin{pmatrix}
0 & I_{F_2}\\
UKU^{-1} & 0
\end{pmatrix},$$
hence ${\bf UKU}^{-1}$ generates a strongly continuous semigroup on ${\bf E}_2$ if and only if $UKU^{-1}$ generates a cosine operator function with associated phase space $F_2\times E_2$. Finally,~\eqref{similrel} follows from~\eqref{groupcos} and the known relation 
$${\bf U}e^{t{\bf K}}{\bf U}^{-1} = e^{t{\bf UKU}^{-1}},\qquad t \geq 0,$$ 
cf.~\cite[\S~II.2.1]{[EN00]}.
\end{proof}

We will need the following perturbation lemma which improves a result due to Piskar{\"e}v and Shaw, cf. \cite[pag. 232]{[PS95]}.

\begin{lemma}\label{pert}
Let $K$ generate a cosine operator function with associated phase space $F\times E$, and let $J$ be a bounded operator from $[D(K)]$ to $F$. Then also $K+J$ generates a cosine operator function with associated phase space $F\times E$. 
\end{lemma} 

\begin{proof}
The operator matrix ${\bf K}$ defined in~\eqref{reduxm} generates a strongly continuous semigroup on $F\times E$. Its perturbation 
$${\bf J}:=\begin{pmatrix}
0 & 0\\
J & 0
\end{pmatrix}$$
is a bounded operator on the Banach space $[D({\bf K})]=[D(K)]\times F$. By a well-known perturbation result due to Desch and Schappacher (see~\cite{[DS84]}), also their sum 
$${\bf K+J}=\begin{pmatrix}
0 & I_{E_1}\\
K+J & 0
\end{pmatrix}$$
generates a strongly continuous semigroup on $F\times E$, that is, $K+J$ generates a cosine operator function with associated phase space $F\times E$.
\end{proof}

\begin{rem}\label{damp}
With a proof that is analogous to that of Lemma~\ref{pert} we also obtain that the reduction matrix
$$\begin{pmatrix}
0 & I_{F}\\
K & H
\end{pmatrix}$$
generates a strongly continuous semigroup on $F\times E$, provided that $K$ generates a cosine operator function with associated space $F\times E$, and that $H\in{\mathcal L}(F)$. This is equivalent to saying that the initial value problem associated with 
$$\ddot{u}(t)=Ku(t)+H\dot{u}(t),\qquad t\in{\mathbb R},$$
is well-posed. In particular, the unboundedness of the damping term $H$ does not prevent backward solvability of the equation.\end{rem}

\section{General framework and basic results}

Inspired by \cite{[CENN03]} and \cite{[KMN03]}, we impose the following
assumptions throughout this paper.

\begin{assum}\label{ass}
{\hskip0em
\begin{enumerate}
\item $X$, $Y$, and ${\partial X}$ are Banach spaces
such that $Y\hookrightarrow X$.
\item ${A}:D({A})\to X$ is linear, with
$D(A)\subset Y$.
\item $L:D(A)\to {\partial X}$ is linear and surjective.
\item ${A}_0:={A}\arrowvert_{\ker(L)}$ is closed, densely
defined, and has nonempty resolvent set.
\item $\begin{pmatrix}A\\L\end{pmatrix}:D(A)\to X\times \partial X$ is closed.
\end{enumerate}}
\end{assum}

The Assumptions~\ref{ass} allow us to state a slight modification of a result due to Greiner, cf. \cite[Lemma~2.3]{[CENN03]} and \cite[Lemma~3.2]{[Mu04]}. 

\begin{lemma}\label{lemmamu}
Let $\lambda\in\rho(A_0)$. Then the restriction $L\big\arrowvert_{\ker(\lambda-{A})}$ has an inverse
$$D^A_\lambda :\partial X\to\ker(\lambda-{A}),$$
called Dirichlet operator associated with $A$. Moreover, $D^A_\lambda$ is bounded from ${\partial X}$ to $Z$, for every Banach space $Z$ satisfying $D(A^\infty)\subset Z\hookrightarrow X$. In particular\footnote{By assumption, $\begin{pmatrix}A\\L\end{pmatrix}$ is a closed operator, thus its domain $D(A)$ endowed with the graph norm becomes a Banach space. We denote it by $[D(A)_L]$.}, 
$D^A_\lambda\in{\mathcal L}(\partial X,[D(A)_L])$ and $D^A_\lambda\in{\mathcal L}(\partial X,Y)$.
\end{lemma}

By the following we precise what kind of {\sl feedback} operators $B$ and $\tilde B$ we allow.

\begin{assum}\label{assbis}
{\hskip0em
\begin{enumerate}
\item $B:[D(A)_L]\to\partial X$ is linear and bounded. 
\item $\tilde{B}:{\partial X}\to{\partial X}$ is linear and bounded.
\end{enumerate}}
\end{assum}

Observe that Assumption~\ref{assbis}.(1) implies that $B$ is relatively $A_0$-bounded. Moreover, by Lemma~\ref{lemmamu} we obtain the following.  

\begin{lemma}\label{lemmino}
Let $\lambda\in\rho(A_0)$. Then the operator $BD^A_\lambda$ is bounded on $\partial X$ and the operator $D^A_\lambda B$ is bounded from $[D(A_0)]$ to $Y$.
\end{lemma}

For $\lambda\in\rho(A_0)$ we denote in the following by $B_\lambda$ the operator 
$$B_\lambda:=\tilde{B}+BD^A_\lambda,$$ 
which by the above result is bounded on $\partial X$.

\bigskip
To start our investigations on $\rm(AIBVP_2)$, we re-write such a problem as a more usual second order abstract Cauchy problem
\begin{equation}\tag{$\mathcal{ACP}_2$}
\left\{
\begin{array}{rcll}
 \ddot{\mathfrak u}(t)&=& \mathcal{A}\mathfrak{u}(t), &t\in{\mathbb R},\\
  {\mathfrak u}(0)&=&\mathfrak{f},\qquad  \dot{\mathfrak u}(0)=\mathfrak{g},&
\end{array}
\right.
\end{equation}
where
\begin{equation}\label{acorsivo}
\mathcal{A}:=\begin{pmatrix}
A & 0\\
B & \tilde{B}
\end{pmatrix},\qquad D(\mathcal{A}):=\left\{ 
\begin{pmatrix}
u\\ x
\end{pmatrix}\in D(A)\times\partial X : Lu=x
\right\},
\end{equation}
is an operator matrix with non-diagonal domain on the product space 
$$\mathcal{X}:=X\times {\partial X}.$$
Here 
\begin{equation}\label{initdata}
\mathfrak{u}(t):=\begin{pmatrix}
u(t)\\ Lu(t)
\end{pmatrix}\quad\quad\hbox{for}\quad t\in{\mathbb R},\qquad\quad 
\mathfrak{f}:=\begin{pmatrix}
f\\ h
\end{pmatrix},\quad 
\mathfrak{g}:=\begin{pmatrix}
g\\ j
\end{pmatrix}.
\end{equation}

\bigskip
We are interested in well-posedness of $({\rm AIBVP}_2)$ in the following sense.

\begin{defi}\label{defiaibvp2}
A \emph{classical solution to} $({\rm AIBVP}_2)$ is a function $u$ such that
\begin{itemize}
\item $u(\cdot)\in C^2({\mathbb R},X)\cap C^1({\mathbb R},Y)$,
\item $u(t)\in D(A)$ for all $t\in {\mathbb R}$, 
\item $Lu(\cdot)\in C^2({\mathbb R},{\partial X})$, and
\item $u(\cdot)$ satisfies $({\rm AIBVP}_2)$.
\end{itemize}
The problem $(\rm{AIBVP}_2)$ is called \emph{well-posed} if it admits a unique classical solution $u$ for all initial data $f\in D(A)$, $g\in Y$, and $h,j\in \partial X$ satisfying the compatibility condition $Lf=h$, and if the dependence of $u$ on $f,g,h,j$ is continuous.
\end{defi}

\begin{rem}\label{wellpo}
One can easily check that $({\rm AIBVP}_2)$ is well-posed if and only if $(\mathcal{ACP}_2)$ is well-posed. Thus, by Lemma~\ref{wellp2char} the issue promptly becomes to investigate the operator matrix ${\mathcal A}$ and, in particular, to decide whether it generates a cosine operator function on ${\mathcal X}$, and what is the associated Kisy\'nski space. In fact, let in this case  $\mathfrak{f}$ and $ \mathfrak g$ lie in the domain of $\mathcal A$ and in the associated Kisy\'nski space, respectively. Then, it follows by~\eqref{groupcos} that the unique classical solution to $({\mathcal{ACP}}_2)$ (resp., to $(\rm{AIBVP}_2)$) is given by
\begin{equation}\label{cosres}
\mathfrak{u}(t)=C(t,\mathcal{A})\mathfrak{f} + S(t,\mathcal{A})\mathfrak{g},\qquad t\in\mathbb{R},
\end{equation}
(resp., by the first coordinate of $\mathfrak u$). 

Observe finally that if $\mathfrak f$ defined in~\eqref{initdata} is in $D({\mathcal A})$, then the compatibility condition $Lf=h$ holds.
\end{rem}

\bigskip
Our main Assumptions~\ref{ass} and \ref{assbis} are similar to those imposed in~\cite{[CENN03]} and \cite{[KMN03]} to treat first order problems. The main result obtained in \cite{[KMN03]} was the following. We sketch its proof as a hint for the subsequent investigations.

\begin{lemma}\label{lemmakmn} 
The operator matrix ${\mathcal A}$ defined in~\eqref{acorsivo} generates a strongly continuous semigroup on $\mathcal X$ if and only if the operator $A_0-D^A_\lambda B$ generates a strongly continuous semigroup on $X$ for some $\lambda\in\rho(A_0)$.
\end{lemma} 

\begin{proof}
The main idea of the proof is that under our assumptions the factorisation
\begin{equation}\label{fac}
{\mathcal A}-\lambda={\mathcal A}_\lambda {\mathcal L}_\lambda:=
\begin{pmatrix}
A_0-\lambda & 0\\
B & B_\lambda-\lambda
\end{pmatrix}
\begin{pmatrix}
I_X & -D^A_\lambda\\
0 & I_{\partial X}
\end{pmatrix}
\end{equation}
holds for all $\lambda\in\rho(A_0)$, cf.~\cite[Lemma~4.2]{[KMN03]}. Since the operator matrix ${\mathcal L}_\lambda$ is an isomorphism on $X\times \partial X$, we obtain by similarity that ${\mathcal A}-\lambda$, and hence $ \mathcal A$ are generators on ${\mathcal X}$ if and only if
\begin{equation}\label{fac2}
{\mathcal L}_\lambda{\mathcal A}_\lambda=
\begin{pmatrix}
A_0-D^A_\lambda B & 0\\
0 & 0\\
\end{pmatrix}+
\begin{pmatrix}
0 & 0\\
B & 0\\
\end{pmatrix}+
\begin{pmatrix}
-\lambda & D^A_\lambda(\lambda-B_\lambda)\\
0 & B_\lambda-\lambda
\end{pmatrix}
\end{equation}
with diagonal domain $D({\mathcal L}_\lambda{\mathcal A}_\lambda)=D(A_0)\times \partial X$ is a generator on $\mathcal X$. Since $B$ is relatively $A_0$-bounded, the second operator on the right-hand side is bounded on $[D({\mathcal L}_\lambda{\mathcal A}_\lambda)]=[D(A_0)]\times \partial X$, and the third one is bounded on $\mathcal X$ as a direct consequence of Lemma~\ref{lemmamu} and ~\ref{lemmino}. Taking into account the already mentioned
perturbation result due to Desch and Schappacher (see~\cite{[DS84]}) the claim follows.
\end{proof}

\section{The case $L\not\in{\mathcal L}(Y,\partial X)$}

If $A_0$ generates a cosine operator function with associated phase space $Y\times X$, then it is intuitive to consider the product space 
$${\mathcal Y}:=Y\times \partial X$$
as a candidate Kisy\'nski space for $({\mathcal{ACP}}_2)$. This intuition is partly correct, as we show in this and the next section. 

\bigskip
We can mimic the proof of Lemma~\ref{lemmakmn} and obtain the following.

\begin{theo}\label{main} 
The operator matrix $\mathcal{A}$ generates a cosine operator function with associated phase space
$\mathcal{Y}\times\mathcal{X}$ if and only if ${A}_0$ generates a cosine operator function with associated phase space $Y\times X$.
\end{theo}

\begin{proof}
Let $\lambda\in\rho(A_0)$. Then the operator matrix $\mathcal{A}-\lambda$ is similar to the operator matrix ${\mathcal L}_\lambda{\mathcal A}_\lambda$ defined in~\eqref{fac2}. The similarity transformation is performed by the matrix ${\mathcal L}_\lambda$ introduced in~\eqref{fac}, which is not only an isomorphism on ${\mathcal X}$, but also, by Lemma~\ref{lemmamu}, on~${\mathcal Y}$. Thus, by Lemma~\ref{simil}, ${\mathcal A}$ generates a cosine operator function with associated phase space ${\mathcal Y}\times {\mathcal X}$ if and only if the similar operator ${\mathcal L}_\lambda{\mathcal A}_\lambda$ generates a cosine operator  function with associated phase space ${\mathcal Y}\times {\mathcal X}$. 

We can now factorise ${\mathcal L}_\lambda{\mathcal A}_\lambda$ as in~\eqref{fac2}. Taking into account Lemma~\ref{pert} and the usual bounded perturbation theorem for cosine operator functions, we can finally conclude that ${\mathcal A}$ generates a cosine operator function with associated phase space ${\mathcal Y}\times {\mathcal X}$ if and only if the operator $A_0-D^A_\lambda B$ generates a cosine operator function with associated phase space $Y\times X$. By Lemma~\ref{pert} and Lemma~\ref{lemmino} this is the case if and only if the unperturbed operator $A_0$ generates a cosine operator function with associated phase space $Y\times X$.
\end{proof}

\begin{rem}\label{dampbis} 
By Remark~\ref{damp} we can characterize the well-posedness of 
$$\left\{
\begin{array}{rcll}
\ddot{\mathfrak u}(t)&=& \mathcal{A}{\mathfrak u}(t)+\mathcal{C}\dot{\mathfrak u}(t), &t\in{\mathbb R},\\
  {\mathfrak u}(0)&=&\mathfrak{f},\qquad  \dot{\mathfrak u}(0)=\mathfrak{g},&
\end{array}
\right.$$
for a damping operator ${\mathcal C}\in{\mathcal L}({\mathcal Y})$. The Kisy\'nski space ${\mathcal Y}=Y\times \partial X$ has the nice property that an operator matrix 
$${\mathcal C}:=
\begin{pmatrix}
0 & 0 \\ C & \tilde{C}
\end{pmatrix}$$
is bounded on $\mathcal Y$ if (and only if) $C\in{\mathcal L}(Y,\partial X)$ and $\tilde{C}\in{\mathcal L}(\partial X)$. Thus, we can perturb our dynamical boundary conditions by a quite wide class of \emph{unbounded} (viz, unbounded from $X$ to $\partial X$) damping operators $C$.

More precisely, let $C\in{\mathcal L}(Y,\partial X)$ and $\tilde{C}\in{\mathcal L}(\partial X)$. Then, taking into account Remark~\ref{wellpo}, our approach yields an abstract result that can be reformulated in the following intuitive way: The second order abstract problem 
$$\ddot{u}(t)= {A}u(t), \qquad t\in{\mathbb R},$$
with (damped) dynamical boundary conditions 
$$ (Lu)^{\cdot\cdot}(t)= Bu(t)+ C\dot{u}(t) + \tilde{B} Lu(t) + \tilde{C}(Lu)^{\cdot}(t), \qquad t\in {\mathbb R},$$
has a unique classical solution for all initial conditions
$$u(0)\in D(A),\qquad  \dot{u}(0)\in Y,\qquad{\rm and}\quad (Lu)^\cdot(0)\in\partial X,$$
depending continuously on the initial values, if and only if the same problem with homogeneous boundary conditions 
$$Lu(t) =0,\qquad t\in{\mathbb R},$$
has a unique classical solution for all initial conditions 
$$u(0)\in D(A)\qquad{\rm and}\qquad {\dot u}(0)\in Y,$$
depending continuously on the initial values. We will consider a concrete example of such damped boundary conditions in Proposition~\ref{appl2}.
\end{rem}

\bigskip
In a simple case we can express $(C(t,{\mathcal A}))_{t\in{\mathbb R}}$ in terms of $(C(t,A_0))_{t\in{\mathbb R}}$. This is relevant to obtain solutions to inhomogeneous problems, and parallels an analogous expression obtained for $(e^{t{\mathcal A}})_{t\geq 0}$ in \cite[Thm.~3.6]{[KMN03]}. (As an example of a setting where the following result holds we mention the case where $A_0$ is in fact the Laplacian equipped with Robin boundary conditions.)

\begin{cor}\label{formula}
Assume $A_0$ to be invertible and to generate a cosine operator function on $X$. Let further $B=\tilde{B}=0$. Then 
$${\mathcal A}=\begin{pmatrix}A & 0\\ 0 & 0\end{pmatrix},\qquad
D({\mathcal A})=\left\{\begin{pmatrix}
u\\ x
\end{pmatrix}\in D(A)\times\partial X : Lu=x
\right\},$$
generates a cosine operator function on $\mathcal X=X\times \partial X$ which is given by
\begin{equation}\label{expl}
C(t,{\mathcal A})=\begin{pmatrix}
C(t,A_0) & D_0-C(t,A_0)D_0
\\ 0 & I_{\partial X}
\end{pmatrix},\qquad t\in{\mathbb R}.
\end{equation}
\end{cor}

\begin{proof}
It has been shown in the proof of Lemma~\ref{lemmakmn} that the operator matrix $\mathcal A$ is similar to ${\mathcal L}_0 {\mathcal A}_0$ given by \eqref{fac2}, that is, to 
$${\mathcal L}_0{\mathcal A}_0=
\begin{pmatrix}
A_0 & 0\\
0 & 0
\end{pmatrix}.$$
Now, ${\mathcal L}_0 {\mathcal A}_0$ is a diagonal operator matrix whose entries generate cosine operator functions. Thus, also ${\mathcal L}_0 {\mathcal A}_0$ generates a cosine operator function  that is given by
$$C(t,{\mathcal L}_0 {\mathcal A}_0)=
\begin{pmatrix}
C(t,A_0) & 0\\
0 & I_{\partial X}
\end{pmatrix},\qquad t\in\mathbb{R}.$$
Applying Lemma~\ref{simil} we obtain that $(C(t,{\mathcal A}))_{t\in{\mathbb R}}=({\mathcal L}_0^{-1} C(t,{\mathcal L}_0 {\mathcal A}_0){\mathcal L}_0)_{t\in\mathbb R}$ is given by~\eqref{expl}.
\end{proof}

\begin{rems}\label{unbound} 
1. Although the setting in which it holds is elementary, the above corollary bears some interest in that one can easily check the relation between the boundedness of the cosine operator function generated by $A_0$ and the boundedness of the cosine function generated by $\mathcal A$. More precisely, under the assumptions of Corollary~\ref{formula} it follows by~\eqref{expl} that $(C(t,{\mathcal A}))_{t\in\mathbb R}$ is bounded (resp., $\gamma$-periodic) on $\mathcal X$ if and only if $(C(t,A_0))_{t\in\mathbb R}$ is bounded (resp., $\gamma$-periodic).  On the other hand, integrating~\eqref{expl} one sees that the associated sine operator function is
\begin{equation}\label{formsine}
S(t,{\mathcal A})=\begin{pmatrix}
S(t,A_0) & tD_0-S(t,A_0)D_0
\\ 0 & tI_{\partial X}
\end{pmatrix},\qquad t\in{\mathbb R}.
\end{equation}
This shows that, under the assumptions of Corollary~\ref{formula}, $(S(t,{\mathcal A}))_{t\in\mathbb R}$ is never bounded on $\mathcal X$, be $(S(t,A_0))_{t\in\mathbb R}$ (or, equivalently, $(C(t,A_0))_{t\in\mathbb R}$) bounded or not. However, $(S(t,{\mathcal A}))_{t\in\mathbb R}$ is indeed bounded (resp., $\gamma$-periodic) on ${\rm ker}(L)\times \{0\}$ if and only if $(C(t,A_0))_{t\in\mathbb R}$ is bounded (resp., $\gamma$-periodic).

2. The above results yield in particular that the abstract wave equation with inhomogeneous boundary conditions
$$\left\{
\begin{array}{rcll}
 \ddot{u}(t)&=& {A}u(t), &t\in{\mathbb R},\\
 Lu(t)&=& \psi t +\xi, &t\in {\mathbb R},\\
   u(0)&=&f,\qquad  \dot{u}(0)=g,\qquad (Lu)^\cdot(0)=j,&
\end{array}
\right.$$
has a unique classical solution for all $\psi,\xi\in \partial X$ and all $f\in D(A)$, $g\in Y$, and $j\in\partial X$, depending continuously on the initial data, if and only if $A_0$ generates a cosine operator function with associated phase space $Y\times X$. Let now $A_0$ be invertible. Then, by~\eqref{formsine} such a classical solution is necessarily unbounded whenever $j\not=0$; on the other hand, for $j=0$ the solution to the above highly non-dissipative inhomogeneous problem is bounded if and only if $(C(t,A_0))_{t\in\mathbb R}$ is bounded.
\end{rems}

\section{The case $L\in{\mathcal L}(Y,\partial X)$}

We now consider the case where the boundary operator $L$ is bounded from the Kisy\'nski space to the boundary space. As already mentioned in Section~1, this case needs to be treated differently. To this aim, we complement the Assumptions~\ref{ass} and~\ref{assbis} by the following, which we impose
throughout this section.

\begin{assum}\label{ass2}
{\hskip0em
\begin{enumerate}
\item $V$ is a Banach space such that $V\hookrightarrow Y$.
\item $L$ can be extended to an operator that is bounded from $Y$ to $\partial X$, which we denote again by $L$, and such that $\ker(L)=V$.
\end{enumerate}}
\end{assum}

To adapt the methods of Section~4 to the current setting, we need to apply Lemma~\ref{simil}. This is made possible by the following.

\begin{lemma}\label{mainlemma}
Consider the Banach space 
\begin{equation}\label{spacey}
{\mathcal V}:=\left\{
\begin{pmatrix} u\\ x\end{pmatrix}\in
 Y\times {\partial X} : Lu=x\right\}.
\end{equation}
Then for all $\lambda\in\rho(A_0)$ the operator matrix ${\mathcal L}_\lambda$ defined in~\eqref{fac} can be restricted to an operator matrix that is an isomorphism from ${\mathcal V}$ to 
$${\mathcal W}:=V\times \partial X,$$ 
which we denote again by ${\mathcal L}_\lambda$. Its inverse is the operator matrix
\begin{equation}
\begin{pmatrix}\label{inv}
I_V & D^A_\lambda\\
0 & I_{\partial X}
\end{pmatrix}.
\end{equation}
\end{lemma}

\begin{proof} Take $\lambda\in\rho(A_0)$. The operator matrix ${\mathcal L}_\lambda$ is everywhere defined on ${\mathcal V}$, and for $\mathfrak{u}=\begin{pmatrix}u \\ Lu\end{pmatrix}\in\mathcal V$ there holds
$${\mathcal L}_\lambda{\mathfrak u}=
\begin{pmatrix}
I_Y & -D^A_\lambda\\
0 & I_{\partial X}
\end{pmatrix}
\begin{pmatrix}
u \\ Lu
\end{pmatrix}=
\begin{pmatrix}
u-D^A_\lambda Lu\\ Lu
\end{pmatrix}.$$
Now $u\in Y$ and also $D^A_\lambda Lu\in Y$, due to Lemma~\ref{lemmamu}. Thus, the vector $u-D^A_\lambda Lu\in V$, since also $L(u-D^A_\lambda Lu)=Lu- LD^A_\lambda Lu= Lu-Lu=0$. This shows that ${\mathcal L}_\lambda{\mathfrak u}\in {\mathcal W}$. 

Moreover, one sees that the operator matrix given in~\eqref{inv} is the inverse of ${\mathcal L}_\lambda$. To show that it maps ${\mathcal W}$ into ${\mathcal V}$, take $v\in V$, $x\in\partial X$. Then
$$\begin{pmatrix}
I_V & D^A_\lambda\\
0 & I_{\partial X}
\end{pmatrix}
\begin{pmatrix}
v \\ x
\end{pmatrix}=
\begin{pmatrix}
v+D^A_\lambda x\\ x
\end{pmatrix}.$$
Now $v+D^A_\lambda x\in Y$ because $V\hookrightarrow Y$ and due to Lemma~\ref{lemmamu}. Moreover, $Lv=0$ by definition of $V$, thus $L(v+D^A_\lambda x)=LD^A_\lambda x=x$, and this yields the claim.
\end{proof}

\begin{theo}\label{main2} 
The operator matrix $\mathcal{A}$ generates a cosine operator function with associated phase space
${\mathcal V}\times{\mathcal X}$ if and only if ${A}_0-D^A_\lambda B$ generates a cosine operator function with associated phase space $V\times X$ for any $\lambda\in\rho(A_0)$.
\end{theo}

\begin{proof}
The proof essentially mimics that of Theorem~\ref{main}. We need to observe that, by Lemma~\ref{simil} and Lemma~\ref{mainlemma}, $\mathcal A$ generates a cosine operator function with associated phase space ${\mathcal V}\times {\mathcal X}$ if and only if the operator matrix ${\mathcal L}_\lambda {\mathcal A}_\lambda$ defined in~\eqref{fac2} generates a cosine operator
function with associated phase space ${\mathcal W}\times {\mathcal X}$ for some $\lambda\in\rho(A_0)$. Decomposing ${\mathcal L}_\lambda {\mathcal A}_\lambda$ as in~\eqref{fac2} yields the claim.
\end{proof}

\begin{rems}\label{remmain}
1. Checking the proof of Corollary~\ref{formula}, one can see that the Kisy\'nski space plays no role in it. Thus, Corollary~\ref{formula} and Remark~\ref{unbound} hold true also in the setting of this section. In particular, if $A_0$ is invertible and $B={\tilde B}=0$, then $(C(t,{\mathcal A}))_{t\in \mathbb R}$ is bounded (resp., $\gamma$-periodic) on $\mathcal X$ if and only if $(C(t,{A_0}))_{t\in \mathbb R}$ is bounded (resp., $\gamma$-periodic), and in this case also $(S(t,{\mathcal A}))_{t\in \mathbb R}$ is bounded (resp., $\gamma$-periodic) on $V\times\{0\}$.

2. It should be emphasized that the above identification of the Kisy\'nski space 
${\mathcal V}$ is {\sl not} topological, and it may be tricky to endow it with a ``good" norm, since $\mathcal V$ is not a product space. More precisely, the ``natural" norms
$$\left\Vert
{\begin{pmatrix}
u\\ Lu
\end{pmatrix}}
\right\Vert_{\mathcal V}
:= \Vert u\Vert_Y+\Vert Lu\Vert_{\partial X}$$
or (in the Hilbert space case)
$$\left\Vert
{\begin{pmatrix}
u\\ Lu
\end{pmatrix}}
\right\Vert_{\mathcal V}
:= \left(\Vert u\Vert^2_Y+\Vert Lu\Vert^2_{\partial X}\right)^\frac{1}{2}$$
may not be the most suitable -- that is, they may not yield an energy space. This will be made clear in Remark~\ref{motivrevrem}.

3) Observe that if $\mathfrak g$ defined in~\eqref{initdata} is in $\mathcal V$, then the compatibility condition $Lg=j$ holds. Thus, taking into account Remark~\ref{wellpo} Theorem~\ref{main2} can be expressed in the following way: The second order abstract problem 
$$ \ddot{u}(t)= {A}u(t), \qquad t\in{\mathbb R},$$
with dynamical boundary conditions
$$ (Lu)^{\cdot\cdot}(t)= Bu(t)+ \tilde{B} Lu(t),\qquad t\in {\mathbb R},$$
has a unique classical solution for all initial conditions
$$u(0)\in D(A)\qquad{\rm and}\qquad  \dot{u}(0)\in Y,$$
depending continuously on the initial values, if and only if the perturbed second order problem 
$$\ddot{u}(t)= Au(t)-D^A_\lambda Bu(t), \qquad t\in{\mathbb R},$$
with homogeneous boundary conditions 
$$Lu(t)=0,\qquad t\in{\mathbb R},$$
has a unique classical solution for all initial conditions 
$$u(0)\in D(A)\qquad{\rm and}\qquad {\dot u}(0)\in V,$$
depending continuously on the initial values.

4) It follows by Lemma~\ref{lemmino} that $D^A_\lambda B$ is bounded from $[D(A_0)]$ to $Y$ (the Kisy\'nski space in Section~4), while $D^A_\lambda B$ is {\sl not} bounded from $[D(A_0)]$ to the current Kisy\'nski space $V$. In fact, $D(A)$ is in general not contained in $V$, hence we cannot apply Lemma~\ref{lemmamu}. This explains why the characterization obtained in Theorem~\ref{main2} is less satisfactory than that obtained in Theorem~\ref{main}. There the properties of $\mathcal A$   depend exclusively on the properties of the unperturbed operator $A_0$.

Though, in many concrete cases we can still apply some perturbation result if we moreover make some reasonable assumption on the decay of the norm of the Dirichlet operator $D^A_\lambda$ associated to $A$.
\end{rems}

\begin{cor}\label{rhandi} 
Let $A_0$ generate a cosine operator function with associated phase space $V\times X$. Assume that 
\begin{equation}\tag{D}
\Vert D^A_\lambda\Vert_{{\mathcal L}(\partial X,X)} =O(\vert\lambda\vert^{-\epsilon})\quad\quad \hbox{as}\quad\vert\lambda\vert\rightarrow \infty,\quad {\rm Re}\lambda>0,
\end{equation}
and moreover that
\begin{equation}\tag{R}
\int_0^1\Vert B S(s,{A_0})f\Vert_{\partial X}\; ds\;\leq\; M\Vert f\Vert_X
\end{equation}
holds for all $f \in D(A_0)$ and some $M>0$. Then $\mathcal A$ generates a cosine operator function with associated phase space $\mathcal{V}\times\mathcal{X}$.
\end{cor}

\begin{proof} 
Let $\lambda\in\rho(A_0)$. The basic tool for the proof is a general Miyadera--Voigt-type perturbation result due to Rhandi, cf. \cite[Thm.~1.1]{[Rh92]}. In our context, Rhandi's result yields that $A_0-D^A_\lambda B$ generates a cosine operator function with associated phase space $V\times X$ whenever 
$$\int_0^1\Vert D^A_\lambda B S(s,{A_0})f\Vert_{X}\; ds\;\leq\; q\Vert f\Vert_X$$
holds for all $f \in D(A_0)$ and some $q<1$. This condition is clearly satisfied under our assumptions.
\end{proof}

\begin{rem}\label{simpler} 
The assumption on the decay of the norm of the Dirichlet operator that appears in Corollary~\ref{rhandi} is in particular satisfied whenever $D(A)$ is contained in any complex interpolation space $X_\epsilon:=[D(A_0),X]_\epsilon$, $0<\epsilon<1$, cf.~\cite[Lemma~2.4]{[GK91]}. Such interpolation spaces are well defined, since in particular $A_0$ generates an analytic semigroup. Moreover, by~\cite[Prop.~2.2]{[GK91]}  the Dirichlet operators associated with operators that share same domain also enjoy same decay rate.
\end{rem}

\section{Compactness, asymptotic behavior, and regularity}

In this section we investigate compactness, regularity and (almost) periodicity properties for the solutions to $(\rm{AIBVP}_2)$. It is important to observe that, unless otherwise explicitly stated, all the following results hold under the general Assumptions~\ref{ass} and~\ref{assbis}, that is, in both the settings of Sections~4 and 5.

\begin{prop}\label{compres}
The operator matrix $\mathcal {A}$ has compact resolvent if and only if the operator $A_0$ has compact resolvent and $\partial X$ is finite dimensional.
\end{prop}

\begin{proof} Take $\lambda\in\rho(A_0)$ and consider the operator ${\mathcal L}_\lambda$ defined in~\eqref{fac}, which is an isomorphism on ${\mathcal X}$ and maps $D({\mathcal A})$ into 
${\mathcal L}_\lambda D({\mathcal A})=D(A_0)\times\partial X$. 
Since ${\mathcal L}_\lambda$ is not compact (beside in the trivial case of ${\rm dim}\; X<\infty$), by~\cite[Prop.~II.4.25]{[EN00]} the claim follows, because\footnote{Given two Banach spaces $E,F$ such that $F\hookrightarrow E$, $i_{F, E}$ denotes in the following the continuous imbedding of $F$ in $E$.} 
$i_{[D({\mathcal A})],{\mathcal X}}= i_{[D(A_0)]\times \partial X, {\mathcal X}} \circ {\mathcal L}_\lambda$.  
\end{proof}

As already remarked, there are no known abstract, concretely applicable characterizations of bounded cosine operator functions on general Banach spaces. However, \emph{assuming} boundedness of the cosine operator function, which is sometimes known by other means, we can apply the above compactness result and obtain the following. For the notion of almost periodicity we refer to~\cite[\S~4.5]{[ABHN01]}.

\begin{cor}\label{period}
Let $\mathcal A$ generate a bounded cosine operator function. Assume the imbedding $[D(A_0)]\hookrightarrow X$ to be compact, and $\partial X$ to be finite dimensional. Then the following hold.

\begin{enumerate}
\item $(C(t,{\mathcal A}))_{t\in\mathbb R}$ is almost periodic. If further the inclusion
\begin{equation}\label{incl}
P\sigma(A_0)\cup\{\lambda\in\rho(A_0): \lambda\in
  P\sigma(B_\lambda)\}\subset -\left(\frac{2\pi}{\gamma}\right)^2{\mathbb N}^2
\end{equation}
holds for some $\gamma>0$, then $(C(t,{\mathcal A}))_{t\in\mathbb R}$ is in fact periodic with period $\gamma$.
\item If the operators $A_0$ and $B_0$ are both injective, then also $(S(t,{\mathcal A}))_{t\in\mathbb R}$, hence the solutions to $({\mathcal{ACP}}_2)$ are almost periodic. If further the inclusion~\eqref{incl} holds for some $\gamma>0$, then they are in fact periodic with period $\gamma$. 
\end{enumerate}
\end{cor}


\begin{proof} 
To begin with, we need to recall the following result due to Engel, cf.~\cite[\S~2]{[En99]} and~\cite[\S~3]{[Mu04]}: For $\lambda\in\rho(A_0)$ there holds
\begin{equation}\label{charsp}
\lambda\in P\sigma({\mathcal{A}})\iff \lambda\in P\sigma(B_\lambda).
\end{equation}
Moreover, observe that under our assumptions it follows by Lemma~\ref{compres} that $\mathcal A$ has compact resolvent.

1. The almost periodicity of $(C(t,\mathcal{A}))_{t\in\mathbb R}$ is just a corollary of~\cite[Cor.~5.6]{[AB97]}. Further, take into account~\cite[Thm.~1 and Thm.~6]{[Pi83]}. Then, to show the $\gamma$-periodicity of $(C(t,\mathcal{A}))_{t\in\mathbb R}$ it suffices to check that under our assumptions the eigenvalues of $\mathcal A$ lie in $-(\frac{2\pi}{\gamma})^2\mathbb N^2$, for some $\gamma>0$. By~\eqref{charsp}, this holds by assumption.

2. Again by~\eqref{charsp}, if $A_0$ and $B_0$ are both injective, hence invertible, then $\mathcal A$ is invertible, too. It follows by~\cite[Cor.~5.6]{[AB97]} that also $(S(t,\mathcal{A}))_{t\in\mathbb R}$ is (bounded and) almost periodic. We deduce by~\cite[Thm.~2 and Thm.~7]{[Pi83]} that $(S(t,\mathcal{A}))_{t\in\mathbb R}$ is even $\gamma$-periodic if further~\eqref{incl} holds.
\end{proof}

\begin{rems}\label{roots} 
1. By~\cite[Thm.~7]{[Pi83]}, one obtains that $\sigma(A_0)\subset -(\frac{2\pi}{\gamma_0})^2 {\mathbb N}^2$ for some $\gamma_0>0$ if the cosine operator function generated by $A_0$ is periodic with period $\gamma_0$. Thus, condition~\eqref{incl} holds in particular for $\gamma=k\cdot\gamma_0$, $h$ some positive integer, if $(C(t,A_0))_{t\in\mathbb R}$ is $\gamma_0$-periodic and further $\lambda\not\in P\sigma(B_\lambda)$ for all $\lambda\not=-\frac{4\pi^2 h^2}{\gamma^2}$, $h\in\mathbb{N}$.

2. Since under the assumptions of Corollary~\ref{period} 
the operator $B_\lambda$, $\lambda\in\rho(A_0)$, is a scalar matrix, one sees that showing that $\lambda\not\in P\sigma(B_\lambda)$ for all $\lambda\notin -(\frac{2\pi}{\gamma})^2\mathbb{N}^2$ reduces to check that a certain \emph{characteristic equation} has solutions only inside a set of countably many points of the real negative halfline.
\end{rems}

Compactness for cosine operator functions is not a relevant property since it occurs if and only if the underlying Banach space is finite dimensional (see \cite[Lemma~2.1]{[TW77]}). However, the compactness of the associated sine operator function $(S(t))_{t\in\mathbb{R}}$ (i.e., the compactness of $S(t)$ for all $t\in\mathbb{R}$) is less restrictive. By \cite[Prop.~2.3]{[TW77]} we can investigate it by means of Proposition~\ref{compres}.  

\begin{cor}\label{compsine}
Let $\mathcal{A}$ generate a cosine operator function. Then the associated sine operator function $(S(t,\mathcal{A}))_{t\in\mathbb{R}}$ is compact if and only if the imbedding $[D(A_0)]\hookrightarrow X$ is compact and $\partial X$ is finite dimensional.
\end{cor}

By Lemma~\ref{wellp2char} $\mathcal A$ generates a cosine operator function if and only if (a suitable part of) the associated reduction matrix generates a strongly continuous semigroup. Thus, it is sometimes useful to know whether such a reduction matrix has compact resolvent. The following complements a result obtained in~\cite[\S~5]{[GGG03]}.

\begin{lemma}\label{comp} 
The reduction matrix associated with $\mathcal A$ has compact resolvent, i.e., both the imbeddings of $[D({\mathcal A})]$ into the Kisy\'nski space and of the Kisy\'nski space in $\mathcal X$ are compact, if and only if $\partial X$ is finite dimensional and further either of the following holds:
{\hskip0em	
\begin{itemize}
\item $L\not\in{\mathcal L}(Y,\partial X)$ and the imbeddings $[D(A_0)]\hookrightarrow Y\hookrightarrow X$ are both compact, or
\item $L\in{\mathcal L}(Y,\partial X)$ and the imbeddings $[D(A_0)]\hookrightarrow V
\hookrightarrow X$ are both compact.
\end{itemize}}
\end{lemma}

\begin{proof}
Let us assume that the non-trivial case ${\rm dim}\; X=\infty$ holds.

If $L\not\in{\mathcal L}(Y,\partial X)$, then the setting is as in Section~4 and the Kisy\'nski  space associated with $\mathcal A$ is $\mathcal Y$. Take $\lambda\in\rho(A_0)$ and observe that the operator ${\mathcal L}_\lambda$ defined in~\eqref{fac} is an isomorphism on $\mathcal Y$, and that it maps $D({\mathcal A})$ into ${\mathcal L}_\lambda D({\mathcal A})=D(A_0)\times \partial X$.
Since we can decompose 
$$i_{[D({\mathcal A})],{\mathcal Y}}=i_{D(A_0)\times\partial X,{\mathcal Y}}
\circ {\mathcal L}_\lambda,$$ 
the claim follows, as ${\mathcal L}_\lambda$ is not compact.

Let now $L\in{\mathcal L}(Y,\partial X)$. As shown in Section~5, in this case the Kisy\'nski space associated with $\mathcal A$ is $\mathcal V$. Take $\lambda\in\rho(A_0)$. By Lemma~\ref{mainlemma} the operator ${\mathcal L}_\lambda$ is an isomorphism from $\mathcal V$ onto $\mathcal W=V\times\partial X$. Thus, we can decompose 
$$i_{[D({\mathcal A})],{\mathcal V}}=
{\mathcal L}_\lambda^{-1}\circ i_{[D(A_0)]\times \partial X,V\times \partial X}
\circ {\mathcal L}_\lambda.$$ 
Likewise we obtain 
$$i_{{\mathcal V},{\mathcal X}}=
i_{{\mathcal W},{\mathcal X}}\circ {\mathcal L}_\lambda.$$
Since ${\mathcal L}_\lambda$ is not compact, we obtain that $i_{[D({\mathcal A})],{\mathcal V}}$ and 
$i_{{\mathcal V},{\mathcal X}}$ are both compact if and only if $i_{[D(A_0)]\times \partial X,{\mathcal W}}$ and $i_{{\mathcal W},{\mathcal X}}$ are both compact. By definition of the product spaces $\mathcal W$ and $\mathcal X$ the claim follows.
\end{proof}

Finally, we briefly turn to discuss the regularity of the solutions to $(\rm{AIBVP}_2)$.

\begin{prop}\label{reguldyn}
Let $\mathcal A$ generate a cosine operator function. If the initial data $f,g$ belong to 
\begin{equation}\label{regset}
{\mathcal D}^\infty_0:=\bigcap_{h=0}^{\infty}\left\{u\in D(A^\infty): LA^h u=BA^h u=0\right\}
\end{equation}
and moreover $h=j=0$, then the unique classical solution $u=u(t)$ to $({\rm AIBVP}^2)$ belongs to $D(A^\infty)$, for all $t\in\mathbb R$.
\end{prop}

\begin{proof}
It follows by Lemma~\ref{wellp2char} and~\cite[Prop.~II.5.2]{[EN00]} that $C(t,\mathcal{A})$ and $S(t,\mathcal{A})$  map $D(\mathcal{A}^\infty)$ into itself for all $t\in\mathbb R$. One can prove by induction that ${\mathcal D}^\infty_0\times\{0\}\subset D({\mathcal A}^\infty)$. Since $D({\mathcal A}^\infty)\subset D(A^\infty)\times \partial X$, taking into account~\eqref{cosres} the claim follows.
\end{proof}

\section{Applications}

Let us apply the abstract theory developed in Section~5 to a concrete operator. 

\begin{prop}\label{appl1}
Let $p\in[1,\infty)$. Then the operator matrix 
\begin{equation}\label{aappl}
{\mathcal A}:=\begin{pmatrix}
\frac{d^2}{d x^2}+q\frac{d}{dx}+rI & 0\\
\begin{pmatrix}
\alpha_{0} \delta_{0}' \\
\alpha_1 \delta_1'
\end{pmatrix}
& \begin{pmatrix}
\beta_{0} & 0\\
0 & \beta_1
\end{pmatrix}
\end{pmatrix}
\end{equation}
with domain
\begin{equation}\label{aappldom}
D({\mathcal A}):=\left\{\begin{pmatrix}u\\ \begin{pmatrix}x_{0} \\ x_1\end{pmatrix}
\end{pmatrix}\in W^{2,p}(0,1)\times {\mathbb C}^2: u(0)=x_{0},\;u(1)=x_1\right\}
\end{equation}
generates a cosine operator function on ${\mathcal X}:=L^p(0,1)\times {\mathbb C}^2$ for all $q\in L^\infty(0,1)$, $r\in L^p(0,1)$, and $\alpha_{0},\alpha_1,\beta_{0},\beta_1\in{\mathbb C}$. The associated Kisy\'nski space is 
$$\mathcal{V}:=\left\{\begin{pmatrix} 
u\\ \begin{pmatrix}x_{0} \\ x_1\end{pmatrix} 
\end{pmatrix}\in W^{1,p}(0,1)\times {\mathbb C}^2: u(0)=x_{0},\;
u(1)=x_1\right\}.$$ 
Moreover, $\mathcal A$ has compact resolvent, hence the associated sine operator function is compact.
\end{prop}

\begin{proof} 
In order to apply the abstract results of Section~4, we begin by recasting the above problem in an abstract framework. Set 
$$X:=L^p(0,1),\qquad Y:=W^{1,p}(0,1), \qquad {\partial X}:={\mathbb C}^2.$$
We define the linear operators  
$$Au:=u''+qu'+ru\qquad\hbox{for all}\; u\in D({A}):=W^{2,p}(0,1),$$
$$Bu:=\begin{pmatrix}\alpha_{0} u'(0)\\ \alpha_1
u'(1)\end{pmatrix}\qquad\hbox{for all}\; u\in D(B):=D(A),$$
$$Lu:=\begin{pmatrix}u(0)\\ u(1)\end{pmatrix}\qquad\hbox{for all}\; u\in D(L):=Y,$$
$$\tilde{B}:=\begin{pmatrix}
\beta_{0} & 0 \\
0 & \beta_1
\end{pmatrix}.$$
Therefore, we obtain $V=\ker(L)=W^{1,p}_0(0,1)$.

In the following, it will be convenient to write $A$ as the sum
$$A:=A_1+A_2:=\frac{d^2}{d x^2}+\left(q\frac{d}{dx}+rI\right),$$
and to define $A_{1_0}$ and $A_{2_0}$ as the restrictions of $A_1$ and $A_2$, respectively, to 
$$D(A_0):=D(A)\cap \ker(L)=W^{2,p}(0,1)\cap W^{1,p}_0(0,1).$$

The Assumptions~\ref{ass} and~\ref{assbis} have been checked in \cite[\S~9]{[KMN03b]} for the case $p=2$ and for analogous operators, and can be similarly proved for all $p\in[1,\infty)$. Due to the embedding $W^{1,p}(0,1)\hookrightarrow C([0,1])$ the Assumptions~\ref{ass2} are satisfied as well. 

Since $q\in L^\infty(0,1)$ and $r\in L^p(0,1)$, one obtains that $qu'+ru\in L^p(0,1)$ for all $u\in W^{1,p}(0,1)$. Thus, the perturbing operator $A_{2_0}$ is bounded from $V$ to $X$ and we can neglect it. 

On the other hand, the operator $A_{1_0}$ is the second derivative with Dirichlet boundary conditions on $L^p(0,1)$, hence it generates a cosine operator function that, as a consequence of the D'Alembert formula, is given by 
\begin{equation}\label{cosine}
C(t,A_{1_0})f(x)=\frac{{\tilde f}(x+t)+{\tilde f}(x-t)}{2},\qquad t\in{\mathbb R},\; x\in(0,1),
\end{equation}
where $\tilde{f}$ is the function obtained by extending $f\in L^p(0,1)$ first by oddity to $(-1,1)$, and then by 2-periodicity to $\mathbb R$. The associated Kisy\'nski space is $W^{1,p}_0(0,1)$, i.e., $V$. Thus, we can apply Corollary~\ref{rhandi} and obtain that the operator matrix $\mathcal A$ generates a cosine operator function with associated phase space ${\mathcal V}\times {\mathcal X}$ if the conditions (D) (for the Dirichlet operator $D_\lambda^A$ or equivalently, by Remark~\ref{simpler}, for the Dirichlet operator $D^{A_1}_\lambda$ associated with the unperturbed operator $A_1$) and (R) are satisfied.

Solving an ordinary differential equation one can see that $D^{A_1}_\lambda$ is given by
$$(D^{A_{1}}_\lambda y)(s):=\frac{y_1\sinh{\sqrt{\lambda}}(1-s)+ y_2\sinh{\sqrt{\lambda}}(s)}{\sinh{\sqrt{\lambda}}},\qquad y=\begin{pmatrix}
y_1 \\ y_2\end{pmatrix}\in{\mathbb C}^2,\; s\in(0,1),$$
for all $\lambda>0$, and that $\Vert D^{A_{1}}_\lambda\Vert_{{\mathcal L}\left({\mathbb C}^2, L^p(0,1)\right)}=O(\vert\lambda\vert^{-\epsilon})$ as $\lambda\rightarrow +\infty$ if (and only if) $\epsilon< \frac{1}{2p}$, cf.~\cite[\S~2]{[GK91]}. Since by definition $D(A)=D(A_1)$, the same decay rate is enjoyed by the Dirichlet operator associated with $A$. Thus, condition (D) is satisfied.

To check condition (R), observe that integrating~\eqref{cosine} yields that the sine operator function generated by $A_{1_0}$ is given by
$$S(t,A_{1_0})f=\frac{1}{2}\int_{\cdot-t}^{\cdot+t}\tilde{f}(s)ds,\qquad t\in{\mathbb R}.$$
Thus, 
$$BS(t,A_{1_0})f=\frac{1}{2}\begin{pmatrix}\alpha_0\big(\tilde{f}(t)-\tilde{f}(-t)\big)\\
\alpha_1\big(\tilde{f}(1+t)-\tilde{f}(1-t)\big)\\ \end{pmatrix},\qquad t\in\mathbb{R},\; f\in D(A_{1_0}).$$
Since $\tilde{f}$ is by definition the odd, 2-periodic extension of $f$, we see that for $t\in(0,1)$ $\tilde{f}(1+t)=\tilde{f}(-1+t)=-\tilde{f}(1-t)=-f(1-t)$. We conclude that
$$BS(t,A_{1_0})f=\begin{pmatrix}
\alpha_0 f(t)\\ -\alpha_1 f(1-t)
\end{pmatrix},\qquad t\in (0,1),\; f\in D(A_{1_0}).$$
Let $M:=\vert\alpha_{0}\vert +\vert\alpha_1\vert$. Then, 
$$\int_0^1\vert B S(s,{A_0})f\vert\; ds\; = M\int_0^1\vert f(s)\vert\; ds
=M\Vert f\Vert_{L^1(0,1)}\leq M\Vert f\Vert_{L^p(0,1)}$$
for all $f\in D(A_{1_0})$. As already remarked, the perturbation $A_{2_0}$ is bounded from the Kisy\'nski space $V$ to $X$, and we finally conclude that $A_0=A_{1_0}+A_{2_0}$ generates a cosine operator function with associated phase space $V\times X$.

To show that $(S(t,{\mathcal A}))_{t\in{\mathbb R}}$ is a family of compact operators, observe that the Sobolev imbeddings $W^{2,p}(0,1)\cap W^{1,p}_0(0,1) \hookrightarrow W^{1,p}_0(0,1) \hookrightarrow L^p(0,1)$ are compact, hence we can apply Proposition~\ref{compres} and Corollary~\ref{compsine}.
\end{proof}

\begin{rem}\label{semkmn}
Observe that, as a consequence of Proposition~\ref{appl1}, we also obtain that \emph{the operator matrix $\mathcal A$ defined in~\eqref{aappl}--\eqref{aappldom} is the generator of an analytic semigroup of angle $\frac{\pi}{2}$ on} $L^p(0,1)\times {\mathbb C}^2$, $1\leq p<\infty$. An analogous operator matrix $\mathcal A$ on $L^p(\Omega)\times L^p(\partial \Omega)$, $\Omega\subset {\mathbb R}^n$, has also been considered in~\cite{[FGGR02]} (where $A$ is an elliptic operator in divergence form), and by different means in~\cite{[AMPR03]} (where $A=\Delta$). However, the analiticity of the semigroup on $L^1$ has not been proved either in~\cite{[FGGR02]} or in~\cite{[AMPR03]}.
\end{rem}

In view of Proposition~\ref{appl1}, we can tackle a generalization of the second order initial-boundary value problem~\eqref{motiv} and strengthen the statement in Proposition~1.1. We also characterize the periodicity of the solutions to~\eqref{motiv} in terms of the coefficients $\beta_0,\beta_1$ -- although numerically determining those values verifying our condition goes beyond the scope of our paper.

\begin{prop}\label{motivrev}
1. Let $p\in[1,\infty)$. If $q\in L^\infty(0,1)$, $r\in L^p(0,1)$, $\alpha_0,\alpha_1,\beta_0,\beta_1\in \mathbb C$, then the problem
\begin{equation}\label{motivreveq}
\left\{
\begin{array}{rcll}
 \ddot{u}(t,x)&=& u''(t,x)+q(x)u'(t,x)+r(x)u(t,x), &t\in{\mathbb R},\; x\in(0,1),\\
 \ddot{u}(t,j)&=& \alpha_j u'(t,j) + \beta_j u(t,j), &t\in{\mathbb R },\; j=0,1,\\
  u(0,\cdot)&=&f,\qquad  \dot{u}(0,\cdot)=g,&
\end{array}
\right.
\end{equation}
admits for all $f\in W^{2,p}(0,1)$ and $g\in W^{1,p}(0,1)$ a unique classical solution $u$, continuously depending on the initial data. If $f,g\in C^\infty_c([0,1])$, then $u(t)\in C^\infty([0,1])$ for all $t\in\mathbb R$.

2. Let $q\equiv 0$, $r\leq 0$, $\alpha_0=1$, $\alpha_1=-1$, and $(\beta_0, \beta_1)\in{\mathbb R}_-^2 \setminus\{0,0\}$. If $f\in  H^{2}(0,1)$ and $g\in H^{1}(0,1)$, then the solution $u$ to~\eqref{motivreveq} is (with respect to the $L^2$-norm) uniformly bounded in time and almost periodic. If further $r\equiv 0$, then $u$ is in fact periodic with period $2k$ if and only if the roots of the equation
\begin{equation}\label{cr}
\lambda^2+\lambda\left(1+
\frac{2{\sqrt{\lambda}}}{\tanh{\sqrt{\lambda}}}-(\beta_0+\beta_1)\right)
-\frac{{(\beta_0+\beta_1)\sqrt{\lambda}}}{\tanh{\sqrt{\lambda}}}+\beta_0\beta_1=0
\end{equation}
are contained in $-(\frac{\pi}{k})^2 \mathbb{N}^2$ for some positive integer $k$.

3. Let $q=r\equiv 0$ and $\alpha_{0}=\alpha_1=\beta_{0}=\beta_1=0$. If $f\in W^{2,p}(0,1)$ and $g\in W^{1,p}_0(0,1)$, then the solution $u$ to~\eqref{motivreveq} is (with respect to the $L^p$-norm) uniformly bounded in time and periodic with period 2. 
\end{prop}

\begin{proof}
The concrete problem~\eqref{motivreveq} can be re-written in an abstract form as the second order Cauchy problem $({\mathcal{ACP}}_2)$, where $\mathcal A$ is the operator matrix defined in~\eqref{aappl}--\eqref{aappldom}.

1) In view of Remarks~\ref{wellpo} and~\ref{remmain}.(3), the well-posedness of~\eqref{motivreveq} follows directly by Proposition~\ref{appl1}. By Remark~\ref{wellpo}, the unique classical solution $u$ is in fact given by the first coordinate of
$$C(t,\mathcal{A})\begin{pmatrix}f\\ {f(0)\choose f(1)}\end{pmatrix} + S(t,\mathcal{A})\begin{pmatrix}g\\ {g(0)\choose g(1)}\end{pmatrix}, \qquad t\in\mathbb{R}.$$
To check the smoothness of $u$, consider the proof of Proposition~\ref{appl1} and observe that since $D(A)=W^{2,p}(0,1)$, we have $D(A^\infty)=C^\infty([0,1])$. Further, one sees that $C^\infty_c([0,1])$ is contained in ${\mathcal D}^\infty_0$ defined in~\eqref{regset}. The claim now follows by Proposition~\ref{reguldyn}.

2) Let $p=2$, $q\equiv 0$, $r\leq 0$, $\alpha_0=1$, $\alpha_1=-1$, and $(\beta_0, \beta_1)\in{\mathbb R}_-^2 \setminus\{0,0\}$. Then a direct computation shows that $\mathcal A$ is dissipative and symmetric on ${\mathcal X}:=L^2(0,1)\times \mathbb{C}^2$. Moreover, integrating by parts one sees that $\mathcal A$ is injective, and by its resolvent compactness we conclude that $\mathcal A$ is self-adjoint and strictly negative. Hence, by~\cite[Lemma~3.1]{[Go69]} it generates a contractive cosine operator function with associated contractive sine operator function. By Corollary~\ref{period}, $(C(t,{\mathcal A}))_{t\in\mathbb R}$ and $(S(t,{\mathcal A}))_{t\in\mathbb R}$ are almost periodic, too. 

Let now also $r\equiv 0$. Then $A_0$ is the second derivative on $(0,1)$ with Dirichlet boundary conditions, which generates a $2$-periodic cosine operator function (and in fact $\sigma(A_0)=-\pi^2{\mathbb N}^2$). By Remark~\ref{roots}.(2), in order to show that the cosine and sine operator functions generated by $\mathcal A$ are $2k$-periodic it suffices to show that the eigenvalues of the $2\times 2$ matrix $\tilde{B}+BD^A_\lambda$, $\lambda\in\rho(A_0)$, lie in the set $-(\frac{\pi}{k})^2\mathbb{N}^2$. But it has been computed in~\cite[\S~9]{[KMN03b]} that a given $\lambda\in\rho(A_0)$ is an eigenvalue of $\tilde{B}+BD^A_\lambda$ if and only if it is a root of the characteristic equation~\eqref{cr}.

3) Let finally $q=r\equiv 0$ and $\alpha_{0}=\alpha_1=\beta_{0}=\beta_1=0$. Then $A_0$ is the second derivative with Dirichlet boundary conditions, which is invertible and generates on $L^p(0,1)$, $1\leq p<\infty$, the 2-periodic cosine operator function defined in~\eqref{cosine}. Hence, by Remark~\ref{remmain}.(1) we deduce that $(C(t,\mathcal{A}))_{t\in\mathbb R}$ and $(S(t,\mathcal{A}))_{t\in\mathbb R}$ are 2-periodic on $\mathcal X$ and $V\times\{0\}$, respectively. The claim now follows by~\eqref{cosres}.
\end{proof}

\begin{rem}\label{motivrevrem}
Let us consider the setting as in Proposition~\ref{motivrev}.(2). By Proposition~\ref{appl1} we already known that the Kisy\'nski space is 
$${\mathcal V}:=\left\{\begin{pmatrix}u\\ 
\begin{pmatrix}x_0 \\ x_1\end{pmatrix}
\end{pmatrix}\in H^1(0,1)\times {\mathbb C}^2: u(0)=x_0,\; u(1)=x_1\right\}.$$
However, as pointed out in Remark~\ref{remmain}.(2) one still needs to precise which norm endows such a space. One can in fact check that the inner product 
$$\left<\cdot,\cdot\right>_{\mathcal V}:=-\int_0^1 u'(s) \overline{v'(s)} ds +\int_0^1 q(s)u(s)\overline{v(s)} ds +\beta_0u(0)\overline{v(0)} 
+\beta_1u(1)\overline{v(1)},$$
on $\mathcal V$ makes the reduction matrix associated with $\mathcal A$ dissipative on ${\mathcal V}\times\mathcal X$, and therefore it makes $(C(t,{\mathcal A}))_{t\in\mathbb R}$ and $(S(t,{\mathcal A}))_{t\in\mathbb R}$ contractive. Observe that the norm associated with such an inner product on $\mathcal V$ is actually equivalent to the product norm defined by
$$\vert\Vert u\Vert\vert^2 := \Vert u\Vert^2_{H^1(0,1)} + \vert u(0)\vert^2 + \vert u(1)\vert^2.$$
\end{rem}

Let us finally consider a problem that fits into the framework of Section~4, where the boundary variable is not the trace of the inner variable anymore, but rather its normal derivative. 
Observe that the problem~\eqref{DBC} below bears a strong resemblance to what in the literature is called a wave equation with acoustic boundary conditions (see~\cite{[GGG03]} and \cite{[Mu04]}).

\begin{prop}\label{appl2}
Let $p,q,r\in L^\infty(\partial \Omega)$, where $\Omega$ is an open bounded domain of ${\mathbb R}^n$ with boundary $\partial\Omega$ smooth enough. Then the problem
\begin{equation}\label{DBC}
\left\{
\begin{array}{rcll}
 \ddot{u}(t,x)&=& \Delta u, &t\in{\mathbb R},\; x\in\Omega,\\
 {\ddot{\delta}}(t,z)&=& p(z)\dot{u}(t,z)+q(z) \delta(t,z)+r(z)\dot{\delta}(t,z), &t\in{\mathbb R },\; z\in\partial\Omega,\\
  \delta(t,z)&=&{\frac{\partial u}{\partial \nu}}(t,z), &t\in{\mathbb R },\; z\in\partial\Omega,\\
  u(0,\cdot)&=&f, \qquad \dot{u}(0,\cdot)=g,\qquad \dot{\delta}(0,\cdot)=j.&\\
\end{array}
\right.
\end{equation}
admits a unique classical solution for all initial conditions $f\in H^2(\Omega)$, $g\in H^1(\Omega)$, and $j\in L^2(\partial\Omega)$, which depends continuously on them.
\end{prop}

\begin{proof} 
Set 
$$X:=L^2(\Omega),\qquad Y:=H^1(\Omega),\qquad {\partial X}:=L^2(\partial \Omega),$$ 
and define
$$A:=\Delta,\qquad D(A):=\left\{u\in H^\frac{3}{2}(\Omega) : \Delta u\in L^2(\Omega)\right\},$$
$$L:=\frac{\partial}{\partial \nu},\qquad D(L):= D(A),$$
$$B:=0,\qquad \tilde{B}v:=q\!\cdot\! v,\qquad \tilde{C}v:=r\!\cdot\! v,\qquad D(\tilde{B})=D(\tilde{C}):=\partial X,$$
$$(Cu)(z):=p(z)u(z),\qquad \hbox{for all}\; u\in H^1(\Omega),\; z\in\partial \Omega.$$
First consider the undamped case of $p=r\equiv 0$. The Assumptions~\ref{ass} have been checked in the proof of~\cite[Thm.~2.7]{[Mu04]}, while $\Vert\tilde{B}\Vert_{\partial X}=\Vert q\Vert_\infty$, hence also the Assumptions~\ref{assbis} are satisfied. One can see that
$$D(A_0)=\ker(L)=\left\{u\in H^2(\Omega):\frac{\partial u}{\partial \nu}=0\right\},$$
hence $A_0$ is the Laplacian with Neumann boundary conditions, which generates a cosine operator function with associated phase space $H^1(\Omega)\times L^2(\Omega)$ (see~\cite[Thm.~IV.5.1]{[Fa85]}). Hence by Theorem~\ref{main} the operator matrix with coupled domain associated with~\eqref{DBC} generates a cosine operator function with associated phase space $\big(H^1(\Omega)\times L^2(\partial \Omega)\big)\times \big(L^2(\Omega)\times L^2(\partial \Omega)\big)$. By Corollary~\ref{compsine}, the associated sine operator function is compact if and only if $\Omega\subset\mathbb R$ is a bounded interval.
			
For arbitrary $p,r\in L^\infty(\partial \Omega)$, $C$ is as a multiplicative perturbation of the trace operator, which is bounded from $Y=H^1(\Omega)$ to $\partial X=L^2(\partial \Omega)$, while $\tilde C$ is a bounded multiplication operator on $L^2(\partial \Omega)$. By Remark~\ref{dampbis} we finally obtain the well-posedness of~\eqref{DBC}.
\end{proof}

Since the Neumann Laplacian generates a cosine operator function on $L^p(\Omega)$, $\Omega\subset{\mathbb R}^n$, if and only if $p=2$ or $n=1$ (see~\cite{[KW03]}), one sees that the problem~\eqref{DBC} is well-posed in an $L^p$-context if and only if $p=2$ or $n=1$.

\end{document}